\documentclass[11pt, a4paper, fleqn, final]{article}
\usepackage[utf8]{inputenc}
\usepackage[T1]{fontenc}
\usepackage[english]{babel}
\usepackage{csquotes}
\usepackage{amsmath, amsthm, amssymb, mathtools, dsfont}
\usepackage[backend=bibtex, sortcites=true, isbn=false, maxnames=5, mincrossrefs=1, parentracker=true, sorting=nyt]{biblatex}
\usepackage{hyperref, url}
\usepackage{graphicx}
\usepackage{xcolor}
\usepackage{fancyhdr}
\usepackage{libertine}

\allowdisplaybreaks[4]
\newtheorem{theo}{Theorem}
\newtheorem{coro}[theo]{Corollary}
\newtheorem{prop}[theo]{Proposition}

\theoremstyle{definition}

\newtheorem{rema}[theo]{Remark}

\DeclarePairedDelimiter\abs{\lvert}{\rvert}

\DeclarePairedDelimiter{\floor}{\lfloor}{\rfloor}

\DeclareMathOperator{\md}{def} 
\DeclareMathOperator{\np}{p} 
\DeclareMathOperator{\nnp}{np} 

\makeatletter
\def\moverlay{\mathpalette\mov@rlay}
\def\mov@rlay#1#2{\leavevmode\vtop{%
   \baselineskip\z@skip \lineskiplimit-\maxdimen
   \ialign{\hfil$\m@th#1##$\hfil\cr#2\crcr}}}
\newcommand{\charfusion}[3][\mathord]{
    #1{\ifx#1\mathop\vphantom{#2}\fi
        \mathpalette\mov@rlay{#2\cr#3}
      }
    \ifx#1\mathop\expandafter\displaylimits\fi}
\makeatother
\newcommand{\cdotcup}{\charfusion[\mathbin]{\cup}{\cdot}}

\title{Polynomial reconstruction of \\ the matching polynomial\footnote{This work was supported by the National Natural Science Foundation of China.}}
\author{Xueliang Li, Yongtang Shi, Martin Trinks \\[+0.8ex] 
Center for Combinatorics, 
Nankai University, 
Tianjin 300071, China \\[+0.8ex] 
\small \tt lxl@nankai.edu.cn, shi@nankai.edu.cn, \\[-0.4ex] \small \tt martin.trinks@googlemail.com}
\date{}

\begin{document}

\maketitle

\begin{abstract}
The matching polynomial of a graph is the generating function of the numbers of its matchings with respect to their cardinality. A graph polynomial is polynomial reconstructible, if its value for a graph can be determined from its values for the vertex-deleted subgraphs of the same graph. This note discusses the polynomial reconstructibility of the matching polynomial. We collect previous results, prove it for graphs with pendant edges and disprove it for some graphs.
\end{abstract}

\section{Introduction}

The famous (and still unsolved) reconstruction conjecture of \textcite{kelly1957} and \textcite{ulam1960} states that every graph $G$ with at least three vertices can be reconstructed from (the isomorphism classes of) its vertex-deleted subgraphs.

With respect to a graph polynomial $P(G)$, this question may be adapted as follows: Can $P(G)$ of a graph $G = (V, E)$ be reconstructed from the graph polynomials of the vertex deleted-subgraphs, that is from the collection $P(G_{-v})$ for $v \in V$? Here, this problem is considered for the matching polynomial of a graph, which is the generating function of the number of its matchings with respect to their cardinality. 

This paper aims to prove that graphs with pendant edges are polynomial reconstructible and, on the other hand, to display some evidence that arbitrary graphs are not.

In the reminder of this section the necessary definitions and notation are given. Further, the previous results from the literature are mentioned in Section \ref{sec:known_results}. Section \ref{sec:pendant_edges} and Section \ref{sec:counterexamples} contain the result for pendant edges and the counterexamples in the general case.

Let $G = (V, E)$ be a graph. A \emph{matching} in $G$ is an edge subset $A \subseteq E$, such that no two edges have a common vertex. The \emph{matching polynomial} $M(G, x, y)$ is defined as
\begin{align}
M(G, x, y) = \sum_{\substack{A \subseteq E \\ A \text{ is matching in }G}}{x^{\md(G, A)} y^{\abs{A}}},
\end{align}
where $\md(G, A) = \abs{V} - \abs{\bigcup_{e \in A}{e}}$ is the number of vertices not included in any of the edges of $A$. A matching $A$ is a \emph{perfect matching}, if its edges include all vertices, that means if $\md(G, A) = 0$. A \emph{near-perfect matching} $A$ is a matching that includes all vertices except one, that means $\md(G, A) = 1$. For more information about matchings and the matching polynomial, see \cite{farrell1979b, gutman1977, lovasz1986}. 

There are also two versions of univariate matching polynomials defined in the literature, namely the \emph{matching defect polynomial} and the \emph{matching generating polynomial} \cite[Section 8.5]{lovasz1986}. For simple graphs, the previously mentioned matching polynomials are equivalent to each other.

For a graph $G = (V, E)$ with a vertex $v \in V$, $G_{-v}$ is the graph arising from the \emph{deletion} of $v$, i.e. arising by the removal of all edges incident to $v$ and $v$ itself. The multiset of (the isomorphism classes of) the vertex-deleted subgraphs $G_{-v}$ for $v \in V$ is the \emph{deck} of $G$. The \emph{polynomial deck} $\mathcal{D}_P(G)$ 
with respect to a graph polynomial $P(G)$ is the multiset of $P(G_{-v})$ for $v \in V$. A graph polynomial $P(G)$ is \emph{polynomial reconstructible}, if $P(G)$ can be determined from $\mathcal{D}_P(G)$.
\section{Previous results}
\label{sec:known_results}

For results about the polynomial reconstruction of other graph polynomials, see the article by \textcite[Section 1]{bresar2005} and the references therein. For additional results, see \cites{li1995}[Section 7]{tittmann2011}[Subsection 4.7.3]{trinks2012c}.

By arguments analogous to those used in Kelly's Lemma \cite{kelly1957}, the derivative of the matching polynomials of a graph $G = (V, E)$ equals the sum of the polynomials in the corresponding polynomial deck. 

\begin{prop}[Lemma 1 in \cite{farrell1987}]
Let $G = (V, E)$ be a graph. The matching polynomial $M(G, x, y)$ satisfies 
\begin{align}
\frac{\delta}{\delta x} M(G, x, y) = \sum_{v \in V}{M(G_{-v}, x, y)}.
\end{align}
\end{prop}

In other words, all coefficients of the matching polynomial except the one corresponding to the number of perfect matchings can be determined from the polynomial deck and thus also from the deck:
\begin{align}
m_{i, j}(G) = \frac{1}{i} \sum_{v \in V}{m_{i, j}(G_{-v})} \qquad \forall i \leq 1,
\end{align}
where $m_{i, j}(G)$ is the coefficient of the monomial $x^i y^j$ in $M(G,x,y)$. 

Consequently, the (polynomial) reconstruction of the matching polynomial reduces to the determination of the number of perfect matchings.

\begin{prop} \label{prop:polynomial_reconstruction}
The matching polynomial $M(G, x, y)$ of a graph $G$ can be determined from its polynamial deck $\mathcal{D}_M(G)$ and its number of perfect matchings. In particular, the matching polynomials $M(G, x, y)$ of graphs with an odd number of vertices are polynomial reconstructible.
\end{prop}

\textcite[Statement 6.9]{tutte1979} has shown that the number of perfect matchings of a simple graph can be determined from its deck and therefore gave an affirmative answer on the reconstruction problem for the matching polynomial.

The matching polynomial of a simple can graph also be reconstructed from the deck of edge-extracted and edge-deleted subgraphs \cite[Theorem 4 and 6]{farrell1987} and from the polynomial deck of the edge-extracted graphs \cite[Corollary 2.3]{gutman1992}. For a simple graph $G$ on $n$ vertices, the matching polynomial is reconstructible from the collection of induced subgraphs of $G$ with $\floor{\frac{n}{2}} + 1$ vertices \cite[Theorem 4.1]{godsil1981b}.
\section{Result for simple graphs with pendant edges}
\label{sec:pendant_edges}

\begin{theo} \label{theo:forest_perfect_matching}
Let $G = (V, E)$ be a forest. $G$ has a perfect matching if and only if each vertex-deleted subgraph $G_{-v}$ for $v \in V$ has a near-perfect matching.
\end{theo}

\begin{proof}
For the first direction we assume that $G$ has a perfect matching $M$. Then each vertex-deleted subgraph $G_{-v}$ has a near-perfect matching $M' = M \setminus {e}$, where $e$ is the edge in the matching $M$ incident to $v$.

For the second direction, let $w$ be one of the vertices of degree $1$ and $u$ its neighbor. If each vertex-deleted subgraph has a near-perfect matching, say $M'$, so does $G_{-u}$. Hence, $M' \cup \{\{u, w\}\}$ is a perfect matching of $G$.
\end{proof}

Actually, this theorem can be generalized to simple graphs with a pendant edge (or equivalently a vertex of degree $1$).

\begin{theo} \label{theo:pendant_perfect_matching}
Let $G = (V, E)$ be a simple graph with a vertex of degree $1$. $G$ has a perfect matching if and only if each vertex-deleted subgraph $G_{-v}$ for $v \in V$ has a near-perfect matching.
\end{theo}

The proof is exactly the same as for the theorem above. We do not know whether or not this can be further generalized to arbitrary simple connected graphs and are also not aware of publications regarding this questions. Therefore, this problem seems to be worth further studies.

Forests have either none or one perfect matching. Because every pendant edge must be in a perfect matching (in order to cover the vertices of degree $1$) and the same holds recursively for the subforest arising by deleting all the vertices of the pendant edges. Therefore, from Proposition \ref{prop:polynomial_reconstruction} and Theorem \ref{theo:forest_perfect_matching} the polynomial reconstructibility of the matching polynomials follows.

\begin{coro}
The matching polynomials $M(G, x, y)$ of forests are polynomial reconstructible.
\end{coro}

On the other hand, arbitrary graphs with pendant edges can have more than one perfect matching. However, Theorem \ref{theo:forest_perfect_matching} can be extended to obtain the number of perfect matchings. For a graph $G = (V, E)$, the number of perfect matchings and of near-perfect machtings of $G$ is denoted by $\np(G)$ and $\nnp(G)$, respectively.

\begin{theo} \label{theo:pendant_number_perfect_matching}
Let $G = (V, E)$ be a simple graph with a pendant edge $e = \{u, w\}$ where $w$ is a vertex of degree $1$. Then we have
\begin{align}
&\np(G) = \nnp(G_{-u}) \leq \nnp(G_{-v}) \qquad \forall v \in V \text{ and particularly} \\
&\np(G) = \min{\{\nnp(G_{-v}) \mid v \in V\}}.
\end{align}
\end{theo}

\begin{proof}
For each vertex $v \in V$, each perfect matching of $G$ corresponds to a near-perfect matching of $G_{-v}$ (by removing the edge including $v$). But the converse is not necessarily true, namely there are near-perfect matching of $G_{-v}$ leaving a non-neighbor of $v$ in $G$ unmatched. Thus, we have $\np(G) \leq \nnp(G_{-v})$.

In case of the vertex $u$, each near-perfect matching $M'$ of $G_{-u}$ corresponds to a perfect matching $M$ of $G$, namely $M' \cup \{e\}$, and vice versa. Thus, we have $\np(G) = \nnp(G_{-u})$, giving the result.
\end{proof}

By applying this theorem, the number of perfect matchings of a simple graph with pendant edges can be determined from its polynomial deck and the following result is obtained as a corollary.

\begin{coro}
The matching polynomials $M(G, x, y)$ of simple graphs with a pendant edge are polynomial reconstructible.
\end{coro}
\section{Counterexamples for arbitrary graphs}
\label{sec:counterexamples}

While it is true that the matching polynomials of graphs with an odd number of vertices or with an pendant edge are polynomial reconstructible, it does not hold for arbitrary graphs.

There are graphs which have the same polynomial deck and yet their matching polynomials are different. Although there are already counterexamples with as little as six vertices, its seems that nothing have been published before in connection with the question addressed here.

\begin{rema}
The matching polynomials $M(G, x, y)$ of arbitrary graphs are not polynomial reconstructible. The minimal counterexample for simple graphs (with respect to the number of vertices and edges) are the graphs $G_1$, $G_2$ shown in Figure~\ref{fig:counterexample}.
\end{rema}

\begin{figure}
\begin{center}
	\includegraphics{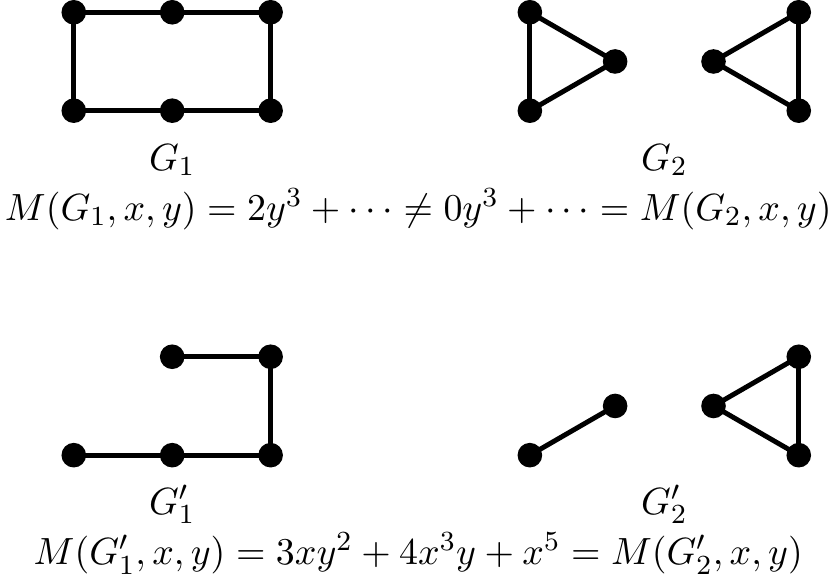}
\end{center}
\caption{Graphs $G_1$ and $G_2$, which are the minimal simple graphs creating a counterexamples for the polynomial reconstructibility of the matching polynomial $M(G, x, y)$. The decks of $G_1$ and $G_2$ consist of six graphs, each isomorphic to $G_1'$ and $G_2'$, respectively. Unlike the matchin polynomials of $G_1$ and $G_2$, the matching polynomials of $G_1'$ and $G_2'$ coincide.}
\label{fig:counterexample}
\end{figure}

The graphs creating the minimal counterexample have six vertices and there are three more pairs of such simple graphs, which are given in Figure \ref{fig:more_counterexamples}.

\begin{figure}
\begin{center}
	\includegraphics{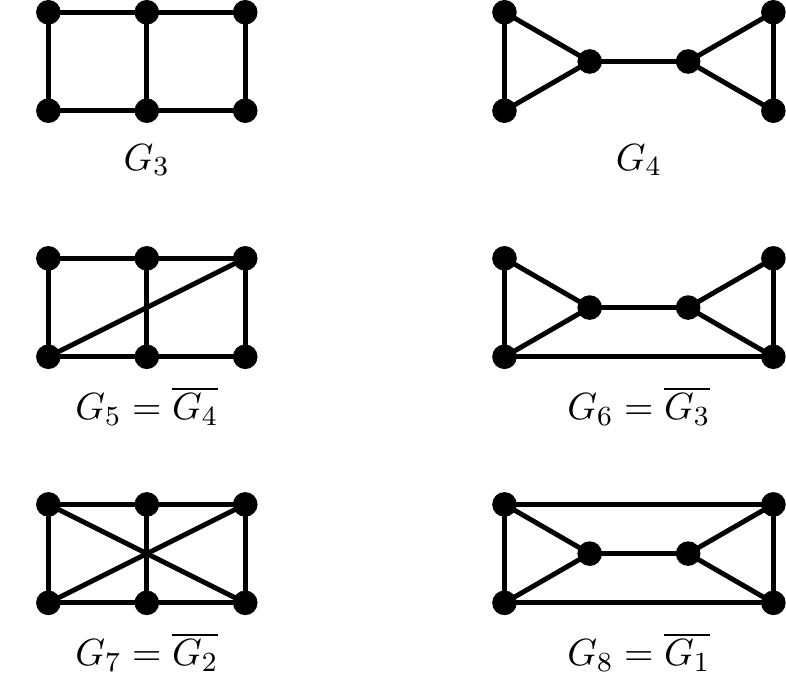}
\end{center}
\caption{The other counterexamples on six vertices for the polynomial reconstructibility of the matching polynomial $M(G, x, y)$.}
\label{fig:more_counterexamples}
\end{figure}

The question arises, whether or not there are such counterexamples consisting of graphs with an arbitrary even number of vertices. In the reminder, we give an affirmative answer to this questions.

Let $P_n$ and $C_n$ be a \emph{path} and a \emph{cycle} on $n$ vertices, respectively. For a graph $G = (V, E)$, $\overline{G}$ denotes the \emph{complement} of $G$, i.e. $\overline{G} = (V, \binom{V}{2} \setminus E)$. For two graphs $G$ and $H$, the \emph{disjoint union} of $G$ and $H$ is denoted by $G \cdotcup H$.

\begin{theo}
Let $k \leq 3$. The matching polynomials $M(G, x, y)$ of the graphs $C_{2k}$, $C_k \cdotcup C_k$ and $\overline{C_{2k}}$, $\overline{C_k \cdotcup C_k}$ are not polynomial reconstructible.
\end{theo}

\begin{proof}
Due to \textcite[Corollary 2.3]{godsil1981c}, the matching polynomial of a graph is determined by the matching polynomial of the complement of this graph. Furthermore, $\overline{G_{-v}} = \overline{G}_{-v}$. Therefore, it is enough to consider the graphs $C_{2k}$ and $C_k \cdotcup C_k$.

The matching polynomials of these two graphs do not coincide  because $C_{2k}$ has exactly two perfect matchings, while $C_k \cdotcup C_k$ has zero ($k$ odd) or four ($k$ even) perfect matchings.

On the other hand, their polynomial decks are identical. At first observe, that $(C_{2k})_{-v}$ is isomorphic to $P_{2k-1}$ and $(C_k \cdotcup C_k)_{-v}$ is isomorphic to $C_k \cdotcup P_{k-1}$ for every vertex $v$ of the respective graph.

It remains to show that the matching polynomials of these graphs in the deck coincide, i.e. $M(P_{2k-1}, x, y) = M(C_k \cdotcup P_{k-1}, x, y)$. Therefore, we make use of the well-known recurrence relation for the matching polynomial \cite[Theorem 1]{farrell1979b}:
\begin{align*}
M(G, x, y) = M(G_{-e}, x, y) + y \cdot M(G_{-u-v}, x, y),
\end{align*}
where $e = \{u, v\}$ is an edge of $G$, $G_{-e}$ is the graph with the edge $e$ deleted and $G_{-u-v}$ is the graph with the vertices of $e$ deleted.

Applying the recurrence relation to the edge connecting the $(k-2)$th and $(k-1)$th vertex of $P_{2k-1}$ (counted from either side), we obtain
\begin{align*}
M(P_{2k-1}, x, y) = M(P_{k-1} \cdotcup P_k, x, y) + y \cdot M(P_{k-2} \cdotcup P_{k-1}, x, y).
\end{align*}
Applying the recurrence relation to an edge of the cycle in $C_k \cdotcup P_{k-1}$, we obtain exactly the same term:
\begin{align*}
M(C_k \cdotcup P_{k-1}, x, y) = M(P_{k} \cdotcup P_{k-1}, x, y) + y \cdot M(P_{k-2} \cdotcup P_{k-1}, x, y).
\end{align*}

It follows, that the polynomial decks coincide, while the matching polynomials of the original graphs do not. Hence, those cannot be determined from the corresponding polynomial decks.
\end{proof}

In fact, the above construction for $k = 2$, in the case of the graphs $C_{4}$ and $C_2 \cdotcup C_2$, where $C_2$ is a graph on two vertices connected by two parallel edges, provide an even smaller counterexample, though the graphs are not simple.

In addition, to obtain examples on an arbitrary even number of vertices such that the graphs and their complements are connected, the construction of the graphs $G_3$ and $G_4$ as well as of their complements $G_5$ and $G_6$ can be generalized analogously.

\section*{Acknowledgement}
Many thanks are due to Julian A. Allagan for his suggestions improving the presentation of this paper.

\printbibliography

\end{document}